\documentclass[12pt]{amsart}
\usepackage[T1]{fontenc}
\usepackage{amssymb}
\usepackage{amsmath}
\usepackage{amsfonts}
\usepackage{geometry}
\usepackage{hyperref}

\hypersetup{
    colorlinks=true,
    linkcolor=black,
    citecolor=black,
    filecolor=black,
    urlcolor=black,
}

\newtheorem{theorem}{Theorem}
\theoremstyle{plain}

\numberwithin{equation}{section}

\newcommand{\R}{ {\mathbb{R}} }

\begin{document}

\title{A stochastic Hamilton-Jacobi equation with infinite speed of propagation}

\author[P. Gassiat]{Paul Gassiat}
\address{Ceremade, Universit\'e de Paris-Dauphine\\
Place du Mar\'echal-de-Lattre-de-Tassigny\\
75775 Paris cedex 16, France}
\email{gassiat@ceremade.dauphine.fr}

\begin{abstract}
We give an example of a stochastic Hamilton-Jacobi equation $du = H(Du) d\xi$ which has an infinite speed of propagation as soon as the driving signal $\xi$ is not of bounded variation.  \end{abstract}

\maketitle

\section{introduction}

An important feature of (deterministic) Hamilton-Jacobi equations
\begin{equation} \label{eq:demHJ}
\partial_t u = H(Du) \;\;\;\; \mbox{ on } (0,T) \times \R^N
\end{equation}
 is the so-called \emph{finite speed of propagation} : assuming for instance that $H: \R^N \to \R$ is $C$-Lipschitz, then if $u^1$ and $u^2$ are two (viscosity) solutions of \eqref{eq:demHJ}, one has 
 \begin{equation} \label{eq:finitespeed}
 u^{1}(0,\cdot) = u^2(0,\cdot) \mbox{ on } B(R) \;\;\Rightarrow \;\;  \forall t \geq 0, \;u^1(t,\cdot) = u^2(t,\cdot) \mbox{ on } B(R-Ct)
 \end{equation}
where by $B(R)$ we mean the ball of radius $R$ centered at $0$.

In this note, we are interested in Hamilton-Jacobi equations with rough time dependence of the form
\begin{equation} \label{eq:stoHJ}
\partial_t u = H(Du) \dot{\xi}(t)  \;\;\;\; \mbox{ on } (0,T) \times \R^N,
\end{equation}
where $\xi$ is only assumed to be continuous. Of course, the above equation only makes classical (viscosity) sense for $\xi$ in  $C^1$, but Lions and Souganidis \cite{LS98b} have shown that if $H$ is the difference of two convex functions, the solution map can be extended continuously (with respect to supremum norm) to any continuous $\xi$. (In typical applications, one wants to take $\xi$ as the realization of a random process such as Brownian motion).

In fact, the Lions-Souganidis theory also gives the following result 
: if  $H=H_1 - H_2$ where $H_1$, $H_2$ are convex, $C$-Lipschitz, with $H_1(0)=H_2(0)=0$, then for any constant $A$,
$$u(0,\cdot) \equiv A \mbox{ on } B(R) \;\;\Rightarrow u(t,\cdot) \equiv A \mbox{ on } B(R(t))$$
where  $R(t)= R - C (\max_{s \in [0,t]} \xi(s) - \min_{s \in [0,t]} \xi(s))$.

However this does not imply a finite speed of propagation for \eqref{eq:stoHJ} for arbitrary initial conditions, and a natural question (as mentioned in lecture notes by Souganidis \cite{S16}) is to know whether a property analogous to \eqref{eq:finitespeed} holds in that case. The purpose of this note is to show that in general it does not : we present an example of an $H$ such that if the total variation of $\xi$ on $[0,T]$ is strictly greater than $R$, one may find initial conditions $u^1_0, u^2_0$ which coincide on $B(R)$, but such that for the associated solutions $u^1$ and $u^2$, one has $u^1(T,0) \neq u^2(T,0)$.

For instance, if $\xi$ is a (realization of a) Brownian motion, then (almost surely), one may find initial conditions coinciding on balls of arbitrary large radii, but such that $u^1(t,0) \neq u^2(t,0)$ for all $t>0$.

It should be noted that the Hamiltonian $H$ in our example is not convex (or concave). When $H$ is convex, some of the oscillations of the path cancel out at the PDE level\footnote{for example : for $\delta \geq 0$, $S_H(\delta)\circ S_{-H}(\delta) \circ S_H(\delta) =S_H(\delta)$, where $S_H, S_{-H}$ are the semigroups associated to $H$, $-H$.}, so that one cannot hope for simple bounds such  as \eqref{eq:main} below. Whether one has finite speed of propagation in this case remains an open question.

\section{main result and proof}

We fix $T>0$ and denote $\mathcal{P} = \left\{ (t_0, \ldots, t_n), \;\; 0= t_0 \leq t_1 \leq \ldots \leq t_n=T \right\}$ the set of partitions of $[0,T]$. Recall that the total variation of a continuous path $\xi :[0,T]\to \R$ is defined by
$$V_{0,T}(\xi) = \sup_{(t_0, \ldots,t_n) \in \mathcal{P}} \sum_{i=0}^{n-1} \left| \xi(t_{i+1}) - \xi(t_i)\right|.$$

Our main result is then :

\begin{theorem} \label{thm:main}
Given $\xi$ $\in$ $C([0,T])$, let $u$ $:$ $[0,T] \times \R^2 \to \R$ be the viscosity solution of
\begin{equation} \label{eq:pde}
\partial_t u = \left( \left|\partial_x u\right| - \left| \partial_y u \right|\right) \dot{\xi}(t) \mbox{ on } (0,T) \times \R^2
\end{equation}
with initial condition
\begin{equation*}
u(0,x,y) = |x-y| + \Theta(x,y)
\end{equation*}
where $\Theta \geq 0$ is such that $\Theta(x,y) \geq 1$ if $\min{x,y}\geq R$.

One then has the estimate 
\begin{equation} \label{eq:main}
u(T,0,0) \geq \left(\sup_{(t_0,\ldots, t_n) \in \mathcal{P}} \frac{\sum_{j=0}^{n-1} \left| \xi(t_{j+1})-\xi(t_j) \right|}{n}  - \frac{R}{n} \right)_+\wedge 1.
\end{equation}
In particular, $u(T,0,0) >0$ as soon as $V_{0,T}(\xi)> R$.
\end{theorem}

Note that since $|x-y|$ is a stationary solution of \eqref{eq:pde}, the claims from the introduction about the speed of propagation follow.

The proof of Theorem \ref{thm:main} is based on the differential game associated to \eqref{eq:pde}. Informally, the system is constituted of a pair $(x,y)$ and the two players take turn controlling $x$ or $y$ depending on the sign of $\dot{\xi}$, with speed up to $|\dot{\xi}|$. The minimizing player wants $x$ and $y$ to be as close as possible to each other, while keeping them smaller than $R$. The idea is then that if the minimizing player keeps $y$ and $x$ stuck together, the maximizing player can lead $x$ and $y$ to be greater than $R$ as long as $V_{0,T}(\xi) > R$.

\begin{proof} [Proof of Theorem \ref{thm:main}]By approximation we can consider $\xi \in C^1$, and in fact we consider the backward equation : 
\begin{equation} \label{eq:pdeB}
\left\{
\begin{array}{ll}
- \partial_t v &= \left( \left|\partial_x v\right| - \left| \partial_y v \right|\right) \dot{\xi}(t), \\
v(T,x,y) &= |x-y| + \Theta(x,y).
\end{array} \right.
\end{equation}
We then need a lower bound on $v(0,0,0)$. Note that
$$\left( \left|\partial_x v\right| - \left| \partial_y v \right|\right) \dot{\xi}(t) = \sup_{|a| \leq 1}  \inf_{|b|\leq 1}  \left\{\dot{\xi}_+(t)  \left(  a \partial_x u  + b \partial_y u\right) + \dot{\xi}_-(t)\left( a \partial_y u  +b \partial_s u\right) \right\},$$
so that by classical results (e.g. \cite{ES84}) one has the representation
 \begin{equation} \label{eq:game}
 v(0,0,0) = \sup_{\delta(\cdot) \in \Delta} \inf_{\beta \in \mathcal{U}} J(\delta(\beta),\beta) ,
 \end{equation}
where $\mathcal{U}$ is the set of controls (measurable functions from $[0,T]$ to $[-1,1]$) and $\Delta$ the set of progressive strategies (i.e. maps $\delta :$ $\mathcal{U} \to \mathcal{U}$ such that if $\beta = \beta'$ a.e. on $[0,t]$, then $\delta(\beta)(t)=\delta(\beta')(t)$). 
Here for $\alpha, \beta$ $\in$ $\mathcal{U}$, the payoff is defined by
$$J(\alpha, \beta) =  |x^{\alpha,\beta}(T)-y^{\alpha,\beta}(T)| + \Theta(x^{\alpha,\beta}(T),y^{\alpha,\beta}(T))$$
where 
$$x^{\alpha,\beta}(0)=y^{\alpha,\beta}(0)=0, \;\; \dot{x}^{\alpha,\beta}(s) = \dot{\xi}_+(s) \alpha(s) + \dot{\xi}_-(s) \beta(s), \;\;\; \dot{y}^{\alpha,\beta}(s) = \dot{\xi}_-(s) \alpha(s) + \dot{\xi}_+(s) \beta(s).$$

Assume $v(0,0,0) < 1$ (otherwise there is nothing to prove) and fix $\varepsilon$ $\in$ $(0,1)$ such that $v(0,0,0)<\varepsilon$. Consider the strategy $\delta^\varepsilon$ for the maximizing player defined as follow : for $\beta \in \mathcal{U}$, let 
$$\tau^\beta_\varepsilon = \inf \left\{t \geq 0, \;\; |x^{1,\beta}(t)-y^{1,\beta}(t)| \geq \varepsilon \right\},$$ 
and then 
$$\delta^\varepsilon(\beta)(t) = \begin{cases} 1, & t < \tau^\beta_\varepsilon \\
\beta(t), & t \geq\tau^\beta_\varepsilon \end{cases}.$$
In other words, the maximizing player moves to the right at maximal speed, until the time when $|x-y|=\varepsilon$, at which point he moves in a way such that $x$ and $y$ stay at distance $\varepsilon$.

Now by \eqref{eq:game}, there exists $\beta \in \mathcal{U}$ with $J(\delta^\varepsilon(\beta), \beta) < \varepsilon$. 
Clearly for the corresponding trajectories $x(\cdot),y(\cdot)$, this means that $|x(T)-y(T)| < \varepsilon$, and  by definition of $\alpha^\varepsilon$ this implies $|x(t)-y(t)| \leq \varepsilon$ for $t\in [0,T]$. We now fix $(t_0,\ldots, t_n) \in \mathcal{P}$ and prove by induction that for $i=0,\ldots, n$,
$$\min\{x(t_i), y(t_i)\} \geq \sum_{j=0}^{i-1} \left| \xi(t_{j+1}) - \xi(t_j) \right| - i \varepsilon.$$
Indeed, if it is true for some index $i$, then assuming that for instance $\xi(t_{i+1}) -\xi(t_i) \geq 0$, one has
\begin{eqnarray*}
x(t_{i+1}) &=& x(t_i) + \int_{t_i}^{t_{i+1}} \dot{\xi}_+(s) ds - \int_{t_i}^{t_{i+1}} \beta(s) \dot{\xi}_-(s) ds \\
 &\geq& x(t_i) + \xi(t_{i+1})-\xi(t_i)  \;\geq \;  \sum_{j=0}^{i}  \left| \xi(t_{j+1}) - \xi(t_j) \right| - i \varepsilon
 \end{eqnarray*}
 and since $y(t_{i+1}) \geq x(t_{i+1}) - \varepsilon$, one also has $y(t_{i+1}) \geq    \sum_{j=0}^{i}  \left| \xi(t_{j+1}) - \xi(t_j) \right|- (i+1) \varepsilon$. The case when $\xi(t_{i+1}) -\xi(t_i) \leq 0$ is similar.
 
 Since $J(\alpha^\varepsilon(\beta), \beta) \leq 1$, one must necessarily have $\min\{x(T),y(T)\} \leq R$, so that
 $$\varepsilon \geq  \frac{1}{n}\left( \sum_{j=0}^{n} \left| \xi(t_{j+1})-\xi(t_j) \right|  - R \right).$$
 Letting $\varepsilon \to v(0,0,0)$ and taking the supremum over $\mathcal{P}$ on the r.h.s. we obtain \eqref{eq:main}.
\end{proof}

\end{document}